\documentclass[12pt,reqno,a4paper]{amsart}
\usepackage{fancyhdr}
\usepackage{enumerate}
\usepackage{amsmath}
\usepackage{amsthm}
\usepackage{mathabx}
\usepackage{amssymb}
\usepackage{pdfsync}
\usepackage[colorlinks=true,linkcolor=blue,citecolor=magenta]{hyperref}

\usepackage[normalem]{ulem}

\usepackage{tikz}
\usetikzlibrary{decorations.pathmorphing,patterns}
\usetikzlibrary{calc,arrows}

\usepackage{xcolor}

\usepackage{changebar}

\makeindex
\newtheorem{theorem}{Theorem}[section]

\newtheorem{proposition}[theorem]{Proposition}
\newtheorem{corollary}[theorem]{Corollary}

\newtheorem{remark}[theorem]{Remark}
\newtheorem{example}[theorem]{Example}

\DeclareMathOperator*{\esslim}{esslim}

\newcommand{\Om}{\Omega}

\newcommand\R{{\mathbb R}}

\newcommand\ve{\varepsilon}
\newcommand\vf{\varphi}

\newcommand \la{\lambda}

\newcommand{\wstar}{\overset\star\rightharpoonup}

\renewcommand{\bar}{\overline}
\renewcommand{\le}{\leqslant}

\renewcommand{\ge}{\geqslant}

\numberwithin{equation}{section}

\begin{document}

\title{Quasi-static Limit for a Hyperbolic Conservation Law}


\author
{Stefano Marchesani}
\address
{Stefano Marchesani: GSSI, 67100 L'Aquila, Italy}
\email
{\tt stefano.marchesani@gssi.it}

\author
{Stefano Olla*}
\address
{Stefano Olla (\emph{corresponding author}):
  CEREMADE, UMR CNRS\\
Universit\'e Paris-Dauphine, PSL Research University\\
75016 Paris, France\\
and GSSI, 67100 L'Aquila, Italy}
\email
{\tt olla@ceremade.dauphine.fr}
\thanks{
  This work was partially supported by ANR-15-CE40-0020-01 grant LSD. We thank Anna De Masi for inspiring discussions and remarks. We thank an anonimous referee whose comments and suggestions helped improve our results and presentation.}

\author
{Lu Xu}
\address
{Lu Xu: GSSI, 67100 L'Aquila, Italy }
\email
{\tt lu.xu@gssi.it}


\begin{abstract}
  We study the quasi-static limit for the $L^\infty$ entropy weak solution
  of scalar one-dimensional hyperbolic equations with strictly concave or convex flux
  and time dependent boundary conditions.
  The quasi-stationary profile evolves with the quasi-static equation,
  whose entropy solution is determined by the stationary profile corresponding to
  the boundary data at a given time. 
\end{abstract}


\maketitle

\section{Introduction}
\label{sec:introduction}

The term \emph{quasi-static evolution}
refers to dynamics driven by external boundary conditions
or forces that change in a time scale much longer than the typical time scale
of the convergence to stationary state of the dynamics.
In the time scale of the changes of the exterior conditions the system
is very close to the corresponding stationary state.
This ideal evolutions are fundamental in Thermodynamics and in many other
situations. We are interested in studying dynamics where the
corresponding quasi-stationary state is of \emph{non-equilibrium}, i.e.
it presents non-vanishing currents of conserved quantities.

In a companion article \cite{DOLXM} we study the quasi-static limit for the
one-dimensional open asymmetric simple exclusion process (ASEP).
The symmetric case was studied in \cite{DO}.
This is a dynamics
where the stationary non-equilibrium states are well studied
\cite{Derrida93,Schutz99,Uchi04}.
The macroscopic equation for the ASEP is given by the traffic flow equation
on the one-dimensional finite interval $[0,1]$:
\begin{equation}
\label{eq:introq0}
  \partial_t u + \partial_xJ\left(u\right) = 0,
\end{equation}
with the flux $J(u) = u(1-u)$,
with time dependent boundary conditions
$u(t,0) = \rho_-(t), u(t,1) = \rho_+(t)$, resulting from the interaction
with external reservoirs. Notice that, 
after the linear transformation $v = 1- 2u$, \eqref{eq:introq0}
is equivalent to Burger's equation
$\partial_t v + \partial_x(\frac{v^2}2) = 0$.
  
For time independent boundary conditions and
a special choice of the dynamics of the reservoirs for the open ASEP,
equation \eqref{eq:introq0}
is obtained as hydrodynamic limit in \cite{Bahadoran12}.
More precisely the hydrodynamic limit generates the $L^\infty$ 
entropy weak solution of  \eqref{eq:introq0}
in the sense of \cite{Otto96}.

Let us consider now the situation when the boundary conditions change
in a slower time scale: for $\ve>0$ small, consider for \eqref{eq:introq0}
the boundary conditions $u(t,0) = \rho_-(\ve t), u(t,1) = \rho_+(\ve t)$.
In order to see the effect of the changes in the boundaries, we need to
look at the evolution in this time scale, i.e. defining
$u^\ve(t,x) = u(\ve^{-1}t,x)$, it will satisfy the equation
\begin{equation}
\label{eq:q0intro}
  \left\{
  \begin{aligned}
    &\ve\partial_t u^\ve + \partial_xJ\left(u^\ve \right) = 0, \quad x\in(0,1), \ t>0,\\
    &u^\ve(t,0) = \rho_-(t), \quad u^\ve(t,1) = \rho_+(t), \quad u^\ve(0,x)=u_0(x).
  \end{aligned}
  \right.
\end{equation}
The main result in this article concerns the convergence of $u^\ve$
to the entropy weak solution of the quasi-static equation
(see section \ref{sec:quasi-stat-equat} for the definition)
\begin{equation}
  \label{eq:QSintro3}
   \partial_xJ (u) = 0, \qquad
    u(t,0) = \rho_-(t), \quad u(t,1) = \rho_+(t).
  \end{equation}
  It turns out that such solutions can only achieve two values
  with at most one upward discontinuity (shock) in the interior of
  the interval $[0,1]$, so they are necessarily of bounded variation
  (see Proposition \ref{BV}). Outside the critical line
  $\{\rho_-(t) + \rho_+(t) = 1, \rho_-(t)<1/2\}$ the solution is unique and
  constant in space (see Proposition \ref{prop:out}).
  On the other hand on the critical line there are infinitely many entropy solutions,
  corresponding to different position of the single shock,
  associated to the same value of the current.
  Consequently we can prove the convergence of $u^\ve$
  to the unique quasi-static solution of the quasi-static equation
  only if $(\rho_-(t), \rho_+(t))$ remains outside the critical line
  for almost every $t$ (see Theorem \ref{thm:quasi}). 
  On the critical line we can only prove the convergence
  to a \emph{measure-valued} solution (cf. Remark \ref{rem-crit}).
  In all cases the quasi-stationary current $\mathcal J(t)$ is constant is space,
  and its value is determined by a variational problem (cf. \eqref{eq:var}):
  the entropy quasi-stationary solution minimize
  $J(\rho) $ when $\rho_-(t) < \rho_+(t)$ (drift up-hill)
  and maximize it when $\rho_-(t) \ge \rho_+(t)$ (drift down-hill).

  Since the ideas contained to this article do not depend on the specific choice
  of the flux $J$, we will expose our results for a generic scalar equation
  \eqref{eq:introq0} with $J(u)$ strictly convex or concave and with $J(0) = 0 = J(u_1)$
  with for some $u_1 >0$. Without losing generality we can set $u_1 =1$ and $J(u)$
  non-negative and strictly concave.

\section{A scalar hyperbolic equation with boundary conditions}
\label{sec:burger}

Consider the following initial--boundary problem for a scalar equation 
on the one-dimensional finite interval $[0,1]$ 
\begin{equation}
\label{eq:b0}
  \left\{
  \begin{aligned}
    &\partial_tv(t,x) + \partial_xJ(v(t,x)) = 0, &&t>0, \ x\in(0,1), \\
    &v(t,0) = \rho_-(t), \ v(t,1) = \rho_+(t), &&t>0, \\
    &v(0,x) = v_0(x), &&x\in[0,1], 
  \end{aligned}
  \right.
\end{equation}
where $\rho_\pm \in L^\infty(\R_+)$ and $v_0\in L^\infty([0,1])$. 
Assume that 
\begin{align}
\label{eq:current}
  J \in C^2(\R), \quad J''<0, \quad J(0)=J(1)=0. 
\end{align}
Also assume that the boundary and initial data are bounded: 
$\rho_\pm(t) \in [0,1]$ for all $t>0$ and $v_0 \in [0,1]$ for almost all $x\in[0,1]$. 
The solution $v \in L^\infty(\R_+\times[0,1])$ is intended in the weak sense: 
for any $\phi \in C_0^\infty(\R\times(0,1))$, 
\begin{equation}
\label{eq:b1}
  \int_0^\infty \int_0^1 \big[v\partial_t\phi + J(v)\partial_x\phi\big]dx\,dt
  + \int_0^1 v_0(x)\phi(0,x)dx = 0. 
\end{equation}
Furthermore, $u$ satisfies the entropy inequality: 
for any $\vf \in C_0^\infty(\R\times(0,1))$ such that $\vf\ge0$, 
\begin{equation}
\label{eq:b2}
\int_0^\infty \int_0^1 \big[S(v)\partial_t\vf + Q(v)\partial_x\vf \big]dx\,dt
+ \int_0^1 S(v_0(x))\vf(0,x)dx \ge 0, 
\end{equation}
where $(S,Q)$ is any pairs of functions such that 
\begin{equation}
\label{eq:entropy}
  S, Q \in C^2(\R), \quad S'' \ge 0, \quad Q' = J'S'. 
\end{equation}
A pair of functions $(S,Q)$ that satisfies \eqref{eq:entropy}
is called a \emph{Lax entropy--entropy flux pair} associated to \eqref{eq:b0}. 
Observe that \eqref{eq:b2} implies the Rankine--Hugoniot jump condition for \eqref{eq:b0}: 
inside the interval eventual discontinuities must be \emph{upwards shocks}. 

Notice that discontinuities can appear at the boundaries. 
The boundary conditions in \eqref{eq:b0} are satisfied in the following sense. 
Assume for the moment that $v(t,\cdot)$ is of bounded variation for each $t$,
so that the limits
\begin{equation*}
  v_-(t) = \lim_{x\to0+} v(t,x), \quad v_+(t) = \lim_{x\to1-} v(t,x) 
\end{equation*}
are well-defined. 
Then the Bardos--LeRoux--N\'ed\'elec boundary conditions \cite{BLN}
of the entropy solution $v$ reads for all $t>0$, 
\begin{equation}
\label{eq:boundary1}
  \mathrm{sign} (v_-(t) - \rho_-(t)) \big[J(v_-(t)) - J(k)\big] \le 0 
\end{equation}
for all $k \in I[v_-(t), \rho_-(t)]$ and 
\begin{equation}
\label{eq:boundary2}
  \mathrm{sign} (v_+(t) - \rho_+(t)) \big[J(v_+(t)) - J(k)\big] \ge 0 
\end{equation}
for all $k \in I[v_+(t), \rho_+(t)]$,
where $I[a,b]$ denotes the closed interval with extremes given by $a$ and $b$. 

Otto in \cite{Otto96} extended the characterization of boundary conditions
to general entropy solutions $v \in L^\infty$ by the use of
\emph{boundary entropy--entropy flux pair}. 
A pair of two-variable functions $(S,Q)$ is called
a \emph{boundary entropy--entropy flux pair}
if $S$, $Q \in C^2(\R^2)$, $(S,Q)(\cdot,w)$
is a entropy--entropy flux pair for each $w\in\R$ and 
\begin{equation}
\label{eq:bd-ent}
  S(w,w) = Q(w,w) = \partial_vS(v,w)|_{v=w} = 0, \quad \forall\:w\in\R. 
\end{equation}
The boundary conditions in \eqref{eq:b0} are then given by 
\begin{equation}
\label{eq:b5}
  \begin{split}
    &\esslim_{r\to0+}  \int_0^\infty Q(v(t,r), \rho_-(t)) \beta(t)dt \le 0,\\
    &\esslim_{r\to0+}  \int_0^\infty Q(v(t,1-r), \rho_+(t)) \beta(t)dt \ge 0, 
  \end{split}
\end{equation}
for any boundary flux $Q$ and $\beta \in C_0(\R)$ such that $\beta\ge0$.
Later on it has been proven that entropy solution of \eqref{eq:b0} has
strong traces at the boundaries even for initial condition in $L^\infty$
(cf. \cite{Vasseur01,Panov05,KV07}), so that the Bardos--LeRoux--N\'ed\'elec
boundary conditions still hold. Nevertheless,
boundary entropy--entropy flux pairs are useful in our proof of the quasi-static limit.

The entropy solution $v$ of \eqref{eq:b0} introduced above can be obtained through the \emph{viscous approximation}. 
For $\delta>0$, let $v^\delta=v^\delta(t,x)$ be the classical solution of the viscous problem 
\begin{equation}
\label{eq:viscous0}
  \left\{
  \begin{aligned}
    &\partial_tv^\delta + \partial_xJ(v^\delta) = \delta\partial_{xx}v^\delta, \quad t>0, \ x\in(0,1), \\
    &v^\delta(\cdot,0) = \rho_-, \quad v^\delta(\cdot,1) = \rho_+, \quad v^\delta(0,\cdot) = v_{0,\delta}, 
  \end{aligned}
  \right.
\end{equation}
where the mollified initial value $v_{0,\delta} \in C^\infty([0,1])$ satisfies that 
\begin{equation}
\label{eq:moll}
  \lim_{\delta\to0+} \int_0^1 |v_{0,\delta}(x)-v_0(x)|dx = 0 
\end{equation}
and the compatibility conditions 
\begin{equation}
  v_{0,\delta}(0,0) = \rho_-(0), \quad v_{0,\delta}(0,1) = \rho_+(0). 
\end{equation}
By \cite[Theorem 8.20]{Malek}, $v^\delta \to v$ in $C([0,T],L^1[0,1])$ for each $T>0$. 

\section{Quasi-static evolution}
\label{sec:quasi}

\subsection{The quasi-static equation}
\label{sec:quasi-stat-equat}

For $\ve>0$, let $u^\ve \in L^\infty(\R_+\times [0,1])$ be the entropy solution of  
\begin{equation}
\label{eq:q0}
  \left\{
  \begin{aligned}
    &\ve\partial_t u^\ve + \partial_xJ(u^\ve) = 0, \\
    &u^\ve(t,0) = \rho_-(t), \quad u^\ve(t,1) = \rho_+(t), \quad u^\ve(0,x) = u_0(x), 
  \end{aligned}
  \right.
\end{equation}
in the sense of \eqref{eq:b1}, \eqref{eq:b2} and \eqref{eq:b5}. 

Our aim is to prove that, as $\ve\to 0$,
the entropy solution $u^\ve$ of \eqref{eq:q0} converge to some $u \in L^\infty$
that is the entropy solution of the \emph{quasi-static conservation law} 
\begin{equation}
\label{eq:q1}
  \partial_xJ(u) = 0, \quad u(t,0) = \rho_-(t), \quad u(t,1) = \rho_+(t). 
\end{equation}
We assume now that $\rho_\pm(t) \in C^1(\R_+)$.
There is a physical reason for such assumption, as this \emph{macroscopic}
changes at the boundaries should be \emph{slow} and \emph{smooth}.
Also we need such condition in the proof of the quasi-static limit
(see proof of Proposition \ref{prop:bd-ent-prod}).

The entropy solution of the quasi-static problem \eqref{eq:q1}
is defined as a function $u \in L^\infty([0,+\infty)\times[0,1])$ such that,
for any $\vf \in C_0^\infty((0,+\infty)\times(0,1))$,
\begin{equation}
\label{eq:wqs}
  \int_0^\infty \int_0^1 J(u)\partial_x\vf\,dx\,dt = 0.
\end{equation}
Furthermore for a flux function $Q$ associated to a convex entropy $S$,
\begin{equation}
\label{eq:q4}
  \int_0^\infty \int_0^1 Q(u)\partial_x\vf\,dx\,dt \ge 0, \quad \forall\,\varphi \in C_0^\infty((0,+\infty)\times(0,1)),\ \varphi\ge0,
\end{equation}
while the boundary conditions are satisfied in the same sense
as in \eqref{eq:b5} with respect to a boundary entropy flux $Q(v,w)$.
Notice the difference with respect \eqref{eq:b2}: quasi-static solutions
are determined by the boundary conditions $\rho_\pm(t)$,
there is no need to specify an initial condition.

Observe from \eqref{eq:current} that the current function $J$ reaches its maximum at some unique $m\in(0,1)$. 
Moreover, for any $y\in[0,J(m)]$ the equation $J(u)=y$ has two solutions:
$u_1(y)\in[0,m]$ and $u_2(y)\in[m,1]$. 

\begin{proposition}
\label{BV}
  Let $u(t,x)$ be $L^\infty$ entropy solution of  \eqref{eq:q1}.
  Then there exists $z_1(t)\in[0,m]$, $z_2(t)\in[m,1]$ such that $J(z_1(t))=J(z_2(t))$ and 
  \begin{equation}
  \label{eq:q3}
    u(t,x) \in \{z_1(t),z_2(t)\}, \quad (t,x) -\text{a.s.}
  \end{equation}
  Furthermore, or $u(t,x)$ is a.s. constant in $x\in(0,1)$ for almost every $t$,
  or there is at most one upward jump from $z_1(t)$ to $z_2(t)$ inside $(0,1)$.
  In particular $u(t,\cdot)$ is of bounded variation for a.e. $t$.
\end{proposition}

\begin{proof}
  Since $u(t,x)$ solves $\partial_xJ(u) = 0$ in the weak sense,
  there exists a bounded function $\mathcal J(t)$ such that
  $J(u(t,x)) = \mathcal J(t)$ almost surely in $(t,x)$. 
  Due to \eqref{eq:current}, we can find $z_1(t) \le m \le z_2(t)$ 
  such that $J(z_1(t)) = J(z_2(t)) = \mathcal J(t)$, and \eqref{eq:q3} thus follows.

  The entropy condition
  \eqref{eq:q4} yields that $\partial_xQ(u(t,x))$
  is negative in the sense of distribution.
  Observe that for any $z_1\in[0,m]$, $z_2\in[m,1]$ s.t. $J(z_1)=J(z_2)=J_0 \in [0,J(m)]$, 
  \begin{align*}
    Q(z_1)-Q(z_2) &= \int_{z_1}^{z_2} Q'(u)du 
    = \int_{z_1}^m S'(u)J'(u)du + \int_m^{z_2} S'(u)J'(u)du \\
    &= \int_{J_0}^{J(m)} S'(u_1(y))dy + \int_{J(m)}^{J_0} S'(u_2(y))dy \\
    &= -\int_{J_0}^{J(m)} \int_{u_1(y)}^{u_2(y)} S''(v)dv\,dy \le 0, 
  \end{align*}
  as $S$ is convex. 
  Hence, only upward jumps from $z_1(t)$ to $z_2(t)$ can decrease 
  the entropy flux $Q(u(t,x))$. 
  This implies that we can have at most one such jump inside $(0,1)$.
\end{proof}

Since, by Proposition \ref{BV},
entropy solution must be of bounded variation,
then the boundary conditions are satified in the Bardos--LeRoux--N\'ed\'elec
sense given by \eqref{eq:boundary1} and \eqref{eq:boundary2}.

For $u\in[0,1]\backslash\{m\}$, let $u^*\in[0,1]\backslash\{m\}$ be such that $J(u^*)=J(u)$. 
Furthermore we fix $u^*=m$ for $u=m$.

Define the \emph{critical segment}
\begin{align}
\label{eq:critical-line}
  \Theta = \{(z,z^*)\in[0,1]^2;\ z<m\}. 
\end{align}
The entropy solution of \eqref{eq:q1} is unique outside $\Theta$
and it can be calculated explicitly as below. 

\begin{proposition}
\label{prop:out}
  Suppose that $(\rho_-(t),\rho_+(t))\notin\Theta$ for almost every $t\ge0$.
  Then \eqref{eq:q1} has a unique entropy solution $u(t,x)$ given by 
  \begin{equation}
  \label{eq:q2}
    u(t,x) = 
    \begin{cases}
      \rho_-(t), &\text{if}\ \rho_-(t) < m,\ \rho_+(t) < \rho_-^*(t), \\
      \rho_+(t), &\text{if}\ \rho_+(t) > m,\ \rho_-(t) > \rho_+^*(t), \\
      m, &\text{if}\ \rho_-(t) \ge m,\ \rho_+(t) \le m. 
    \end{cases}
  \end{equation}
\end{proposition}

\begin{proof}
  We have to specify $z_1(t)$ through the boundary values $\rho_\pm(t)$. 
  From the argument above, $u(t,\cdot)$ has bounded total variation
  for each $t$, hence 
  \begin{equation}
    u_-(t) = \lim_{x\to0+} u(t,x), \quad u_+(t) = \lim_{x\to0-} u(t,x) 
  \end{equation}
  are well-defined. 
  Furthermore, $u_\pm(t) \in \{z_1(t), z_1^*(t)\}$ and $u_-(t) \le u_+(t)$. 
  Rewrite \eqref{eq:boundary1} and \eqref{eq:boundary2} explicitly as 
  \begin{equation}
    \begin{split}
      &\rho_-(t) < m \Rightarrow u_-(t) = \rho_- \text{ or }  u_-(t) \in [\rho_+^*(t), 1], \\
      &\rho_-(t) \ge m \Rightarrow u_-(t) \ge m, \\
      &\rho_+(t) \le m \Rightarrow u_+(t) \le m, \\
      &\rho_+(t) > m \Rightarrow u_+(t) = \rho_+ \text{ or } u_+(t) \in [0, \rho_+^*(t)]. 
    \end{split}
  \end{equation}
If $\rho_- \ge m$, $\rho_+ \le m$, then $u_-=u_+=m$ so that $u(t,x)=m$. 
If $\rho_-<m$, $\rho_+<\rho_-^*$, 
then $u_- \le u_+ \le \max\{m, \rho_+\} < \rho_-^*$, so that $u_-=\rho_-$. 
In view of \eqref{eq:q3}, we have $z_1(t)=\rho_-$ and $u_+=\rho_-$, hence $u(t,x) = \rho_-$. 
The case in which $\rho_+>m$, $\rho_->\rho_+^*$ is proved similarly. 
\end{proof}

\begin{remark}
\label{critical}
  If $(\rho_-(t),\rho_+(t))\in\Theta$ for an interval of time of positive measure,
  then the entropy solution is not unique,
  but for any solution there exists one single shock
  with position $X(t)$ such that $u(t,x) = \rho_-(t)$ for $x < X(t)$
  and $u(t,x) = \rho_+(t) = \rho_-^*(t)$ for $x> X(t)$. 
\end{remark}

\begin{remark}
\label{rem:current}
The entropy solution can also be characterized as the solution of the following
variational problem:
\begin{align}\label{eq:var}
  \mathcal J(t) = 
  \begin{cases}
    \sup\,\{J(\rho); \rho \in [\rho_+(t), \rho_-(t)]\}, &\text{if } \rho_-(t) \ge \rho_+(t), \\
    \inf\,\{J(\rho); \rho \in [\rho_-(t), \rho_+(t)]\}, &\text{if } \rho_-(t) < \rho_+(t). 
  \end{cases}
\end{align}
This also includes the critical line $(\rho_-(t),\rho_+(t)) \in \Theta$,
where $\mathcal J = J(\rho_-) = J(\rho_+)$ minimizes the current $J(\rho)$ in the interval
$[\rho_-, \rho_+]$. 
\end{remark}

\subsection{The quasi-static limit}
\label{sec:quasi-static-limit}

\begin{theorem}
\label{thm:quasi}
Suppose that $\rho_\pm \in C^1(\R_+)$ and
$(\rho_-(t),\rho_+(t)) \notin \Theta$ for almost all $t$, 
then the solution $u^\ve$ of \eqref{eq:q0} converges to $u = u(t,x)$ defined in \eqref{eq:q2}
with respect to the weak-$\star$ topology of $L^\infty([0,T]\times[0,1])$ for all $T>0$. 
\end{theorem}

\begin{remark}
\label{rem:current1}
As $\ve \to 0$, $J(u^\ve(t)) \wstar \mathcal J(t)$ given by \eqref{eq:var}. 
Particularly, in the case $(\rho_-(t),\rho_+(t)) \in \Theta$ we can prove that 
$u^\ve$ converges weakly-$\star$ to a Young measure concentrated on $\{\rho_\pm(t)\}$,  thus $J(u^\ve) \wstar \mathcal J = J(\rho_-) = J(\rho_+)$. 
See Remark \ref{rem-crit} at the end of the section. 
\end{remark}

\begin{remark}
\label{initialcond}
  Notice that the quasi-static limit in Theorem \ref{thm:quasi} does not
  depend on the initial condition $u_0$ for $u^\ve$.
\end{remark}

\begin{example}
  Consider the current function $J(u)=u(1-u)$ in \eqref{eq:introq0}. 
  Proposition \ref{prop:out} and Theorem \ref{thm:quasi} hold in this case with $m=1/2$ and $u^*=1-u$. 
  
  On the other hand, let $v^\delta=v^\delta(t,x)$ be the classical solution of the quasi-static problem associated to the viscous equation \eqref{eq:viscous0}: 
  \begin{equation}
    \partial_xJ(v^\delta) = \delta\partial_{xx}v^\delta, \quad
    v^\delta(t,0)=\rho_-(t), \quad v^\delta(t,1)=\rho_+(t). 
  \end{equation}
  When $(\rho_-(t),\rho_+(t))\notin\Theta$, it is not hard to see that $v^\delta$ also converges pointwisely to the solution $u$ of quasi-static problem given by \eqref{eq:q2}: 
  \begin{equation}
    \lim_{\delta\to0+} v^\delta(t,x) = u(t,x), \quad \forall x\in(0,1), 
  \end{equation}
  and the convergence is uniform on $[\gamma,1-\gamma]$ for any $\gamma>0$. 
  On the critical line $(\rho_-(t),\rho_+(t))\in\Theta$, $v^\delta$ is explicitly given by 
  \begin{equation}
    v^\delta(t,x)=\frac12+\delta C(\delta,t)\tanh \left[ C(\delta,t) \left( x-\frac12 \right) \right], 
  \end{equation}
  where $C=C(\delta,t)$ is such that $C\tanh(C/2)=\delta^{-1}(2\rho_+(t)-1)$. 
  Then $v^\delta$ converges pointwisely to the profile with an upward shock at $1/2$: 
  \begin{equation}
    \lim_{\delta\to0+} v^\delta(t,x) = \rho_-(t)\mathbf1_{[0,\frac12)}(x) + \rho_+(t)\mathbf1_{(\frac12,1]}(x), \quad \forall x\in[0,1], 
  \end{equation}
  and the convergence is uniform on any closed interval excludes $1/2$. 
\end{example}

\section{Proof of Theorem \ref{thm:quasi}}

For $\ve > 0$, $\delta > 0$, consider viscous approximation of \eqref{eq:q0} given by 
\begin{equation}
\label{eq:viscous}
  \left\{
  \begin{aligned}
    &\ve\partial_tu^{\ve,\delta} + \partial_x J(u^{\ve,\delta}) = \delta\partial_{xx}u^{\ve,\delta}, \quad t>0,\,x\in(0,1), \\
    &u^{\ve,\delta}(t,0)=\rho_-(t), \quad u^{\ve,\delta}(t,1)=\rho_+(t),
    \quad u^{\ve,\delta}(0,x)=u_{0,\delta}(x), 
  \end{aligned}
  \right.
\end{equation}
where $u_{0,\delta}$ is the mollified initial function satisfying \eqref{eq:moll} and the compatibility conditions. 
Let $u^{\ve,\delta}=u^{\ve,\delta}(t,x)$ be the classical smooth solution of \eqref{eq:viscous}. 
We first present a priori estimate for $\|\partial_xu^{\ve,\delta}\|_{L^2}$. 

\begin{proposition}
\label{prop:priori}
For any $t \ge 0$, there is a constant $C = C_t$ such that 
\begin{equation}
  \ve\int_0^1 u^{\ve,\delta}(t,x)^2dx + \delta\int_0^t \int_0^1 \big(\partial_x u^{\ve,\delta}(s,x)\big)^2dx\,ds \le C. 
\end{equation}
\end{proposition}

\begin{proof}
Denote by $G(u)$ a primitive of $uJ'(u)$: $G'(u)=uJ'(u)$.
Multiply \eqref{eq:viscous} by $u^{\ve,\delta}$ and integrate over $(0,t)\times(0,1)$ to obtain
\begin{equation}
\label{eq:eninter}
  \begin{aligned}
    &\frac \ve 2  \int_0^1 u^{\ve,\delta}(t,x)^2dx-\frac \ve 2 \int_0^1 u_{0,\delta}(x)^2dx + \int_0^t \big[G(\rho_+(s))- G(\rho_-(s))\big]ds \\
    =\ &\delta\int_0^t \big[\rho_+(s)\partial_x u^{\ve,\delta}(s,1) - \rho_-(s)\partial_x u^{\ve,\delta}(s,0)\big]ds - \delta\iint (\partial_x u^{\ve,\delta})^2dx\,ds. 
  \end{aligned}
\end{equation}
In order to estimate the last line of \eqref{eq:eninter} we test \eqref{eq:viscous} against $\psi(s,x) := \rho_-(s)+x[\rho_+(s)-\rho_-(s)]$, obtaining that 
\begin{equation*}
  \begin{aligned}
    &\ve\int_0^1 \big[\psi(t,x)u^{\ve,\delta}(t,x) - \psi(0,x)u_{0,\delta}(x)\big]dx -\ve \iint u^{\ve,\delta}\partial_s\psi\,dx\,ds \\
    &+\int_0^t \big[J(\rho_+(s))\rho_+(s) - J(\rho_-(s))\rho_-(s)\big]ds - \iint J(u^{\ve,\delta})\partial_x\psi\,dx\,ds\\
    =\ &\delta\int_0^t \big[\rho_+(s) \partial_xu^{\ve,\delta}(s,1)-\rho_-(s) \partial_xu^{\ve,\delta}(s,0)\big] ds - \delta\iint \partial_xu^{\ve,\delta}\partial_x\psi\,dx\,ds. 
  \end{aligned}
\end{equation*}
Then, Young inequality allows to estimate
\begin{equation*}
  \begin{aligned}
    &\left|\delta\int_0^t \big[\rho_+(s)\partial_xu^{\ve,\delta}(s,1)-\rho_-(s) \partial_xu^{\ve,\delta}(s,0)\big]ds \right| \\
    \le\ &C + \frac\ve4\int_0^1u^{\ve,\delta}(t,x)^2dx + \frac\delta2\int_0^t \int_0^1 (\partial_xu^{\ve,\delta}(s,x))^2dxds, 
  \end{aligned}
\end{equation*}
which, inserted into \eqref{eq:eninter} gives the conclusion.
\end{proof}

In the following we denote $\Om_T=[0,T]\times[0,1]$, $\Om=\R_+\times[0,1]$. 
As stated in \S\ref{sec:burger}, for each fixed $\ve>0$, 
\begin{equation*}
  \lim_{\delta\to 0} \iint \vf(t,x)F(t,x,u^{\ve,\delta}(t,x))dx\,dt
  = \iint \vf(t,x)F(t,x,u^\ve(t,x))dx\,dt, 
\end{equation*}
for all $F \in C(\Om_T\times [0,1])$ and $\vf \in L^1(\Om_T)$, 
where $u^\ve \in L^\infty(\Om_T)$ is the entropy solution of \eqref{eq:q0}. 
Observe that $u^\ve$ is uniformly bounded: $\|u^\ve\|_{L^\infty(\Om_T)} \le 1$. 
Therefore, we can extract a weakly-$\star$ convergent subsequence: 
\begin{equation*}
  \lim_{\ve_n\to0} \iint \vf(t,x)F(t,x,u^{\ve_n}(t,x)) dx\,dt =
  \iint \vf(t,x)\int_0^1 F(t,x,\la)\nu_{t,x}(d\la)dx\,dt
\end{equation*} 
where $\{\nu_{t,x}(d\lambda)\}_{(t,x)\in\Om_T}$ is the limit Young measure. 

It suffices to show that $\nu_{t,x}$ coincides with the delta measure concentrated on $u(t,x)$ given by \eqref{eq:q2}. 
To this end, given boundary entropy--entropy flux pair $(S,Q)$, define the \emph{boundary entropy production} 
\begin{equation}
  \bar Q_\pm(t,x) := \int Q(\lambda, \rho_\pm(t))\nu_{t,x}(d\lambda), \quad (t,x) \in \Om_T. 
\end{equation}
The following proposition is the key argument. 

\begin{proposition}
\label{prop:bd-ent-prod}
For any boundary entropy flux $Q$, 
\begin{equation}
  \bar Q_-(t,x) \le 0, \quad \bar Q_+(t,x) \ge 0, \quad (t,x)-\text{a.s.} 
\end{equation}
Moreover, $\partial_x\bar Q_\pm \le 0$ in the sense of distribution. 
\end{proposition}

\begin{proof}
Recall that $u^{\ve,\delta}$ is the classical solution of \eqref{eq:viscous}. 
For $w \in C^1([0, T])$ and boundary entropy--entropy flux $(S,Q)$, 
\begin{align*}
  &\ve\partial_tS(u^{\ve,\delta},w) = \ve\partial_uS(u^{\ve,\delta},w)\partial_tu^{\ve,\delta} + \ve\partial_wS(u^{\ve,\delta},w)w' \\
  =&\:\delta\partial_x^2S(u^{\ve,\delta},w) - \delta\partial_u^2S(u^{\ve,\delta},w)(\partial_xu^{\ve,\delta})^2 - \partial_xQ(u^{\ve,\delta},w) + \ve\partial_wS(u^{\ve,\delta},w)w'. 
\end{align*}
Therefore, for $\vf \in C^\infty(\Om_T)$ such that $\vf(0,x) = \vf(T,x) = 0$, 
\begin{align*}
  &\iint \big[\ve S(u^{\ve,\delta},w)\partial_t\vf + Q(u^{\ve,\delta},w)\partial_x\vf + \ve\partial_wS(u^{\ve,\delta},w)w'\vf\big]dx\,dt \\
  =\ &\delta\iint \big[\partial_xS(u^{\ve,\delta},w)\partial_x\vf + \partial_u^2S(u^{\ve,\delta},w)(\partial_xu^{\ve,\delta})^2\vf\big]dx\,dt \\
  &+ \int_0^T \big[Q(u^{\ve,\delta}(t,1),w(t)) - \delta\partial_xS(u^{\ve,\delta}(t,1),w(t))\big]\vf(t,1)dt \\
  &- \int_0^T \big[Q(u^{\ve,\delta}(t,0),w(t)) - \delta\partial_xS(u^{\ve,\delta}(t,0),w(t))\big]\vf(t,0)dt. 
\end{align*}
Taking $w = \rho_-$, since $u^{\ve,\delta}(\cdot,0) = \rho_-$ and $Q(w,w) = \partial_uS(w,w) = 0$ for all $w \in \R$, the last line above is $0$. 
Hence, choosing $\vf = \vf_+$ such that 
\begin{equation}
\label{eq:boundary test1}
  \vf_+(t, 1) = 0, \quad \vf_+(0, x) = 0, \quad \vf_+(T, x) = 0, 
\end{equation}
we obtain for any convex boundary entropy $S$ that 
\begin{align*}
  &\iint \big[\ve S(u^{\ve,\delta},\rho_-)\partial_t\vf_+ + Q(u^{\ve,\delta},\rho_-)\partial_x\vf_+ + \ve\partial_wS(u^{\ve,\delta},\rho_-)\rho'_-\vf_+\big]dx\,dt \\
  \ge\:&\delta\iint \partial_uS(u^{\ve,\delta},\rho_-)\partial_xu^{\ve,\delta}\partial_x\vf_+dx\,dt. 
\end{align*}
Let $\delta \to 0+$ and apply the priori estimate in Proposition \ref{prop:priori}, 
\begin{equation*}
  \iint \big[\ve S(u^\ve,\rho_-)\partial_t\vf_+ + Q(u^\ve,\rho_-)\partial_x\vf_+ + \ve\partial_wS(u^\ve,\rho_-)\rho'_-\vf_+\big]dx\,dt \ge 0. 
\end{equation*}
Eventually, let $\ve\to0+$ along the convergent subsequence, 
\begin{equation}
  \iint \bar Q_-(t,x)\partial_x\vf_+(t,x)dx\,dt \ge 0. 
\end{equation}
Since this holds for all nonnegative, smooth test function $\vf_+$ satisfying \eqref{eq:boundary test1},  we conclude that 
$\bar Q_- \le 0$ almost everywhere and $\partial_x\bar Q_- \le 0$ as a distribution. 
For $\bar Q_+$, we replace $(\rho_-,\vf_+)$ with $(\rho_+,\vf_-)$ such that 
\begin{equation}
\label{eq:boundary test2}
  \vf_-(t, 0) = 0, \quad \vf_-(0, x) = 0, \quad \vf_-(T, x) = 0, 
\end{equation}
and repeat the same argument. 
\end{proof}

Theorem \ref{thm:quasi} follows directly from the following consequence. 

\begin{corollary}
The followings hold for a.e. $(t,x)$: 
\begin{enumerate}
  \item If $\rho_-(t) < m$, $\rho_+(t) < \rho_-^*(t)$ then $\nu_{t,x} = \delta_{\rho_-(t)}$,
  \item If $\rho_+(t) > m$, $\rho_-(t) > \rho_+^*(t)$ then $\nu_{t,x} = \delta_{\rho_+(t)}$,
  \item If $\rho_-(t) \ge m$, $\rho_+(t) \le m$ then $\nu_{t,x} = \delta_m$,
\end{enumerate}
where for $u\in[0,1]$, $u^*$ is defined above \eqref{eq:critical-line}. 
\end{corollary}

\begin{proof}
Consider the following boundary entropy 
\begin{align*}
  S(u,w) = 
  \begin{cases} 
    w \land m - u, &u \in [0, w \land m), \\ 
    0, &u \in [w \land m, 1]. 
  \end{cases}
\end{align*}
Note that $S$ is not smooth, but it can be approximated by convex, smooth functions easily. 
For instance, let $s \in C^\infty(\R)$ be such that 
\begin{align*}
  s(u) = -u, \ \forall\,u\le-1, \ s(u) = 0, \ \forall\,u\ge1, \ s'' \ge 0. 
\end{align*}
Then $S_a(\cdot,w) \to S(\cdot,w)$ as $a \to 0+$, where 
\begin{align*}
  S_a(u,w) := as\big(a^{-1}(u-w)\big), \quad a>0. 
\end{align*}
The flux corresponding to $S$ is 
\begin{align*}
  Q(u,w) = 
  \begin{cases}
    J(w \land m) - J(u), &u \in [0, w \land m), \\ 
    0, &u \in [w \land m, 1]. 
  \end{cases}
\end{align*}
Since $Q(u,\rho_-) \ge 0$ for all $u \in [0, 1]$ and $\bar Q_- \le 0$, we know that $\nu_{t,x}$ concentrates on its zero set $[\rho_-(t) \land m, 1]$ where $Q(u,\rho_-)=0$. 
A similar argument yields that $\nu_{t,x}$ concentrates on $[0, \rho_+(t) \lor m]$.
Hence, $\nu_{t,x}$ concentrates on 
$$
  I_t = \big[\rho_-(t) \land m, \rho_+(t) \lor m\big]. 
$$
Case 3 follows directly. 
In order to prove case 1 and 2, we choose
$$
  S_*(u,w) = |u-w|, \quad Q_*(u,w) = \text{sign}(u-w)(J(u)-J(w)). 
$$
In case 1, $Q_*(u,\rho_-(t)) \ge 0$ on $I_t$ and the only zero point is $\rho_-(t)$. 
As $\bar Q_- \le 0$, we know that $\nu_{(t,x)} = \delta_{\rho_-}$. 
In Case 2, $Q_*(u,\rho_+(t)) \le 0$ on $I_t$
and the only zero point is $\rho_+(t)$, so the conclusion holds similarly.
\end{proof}

\begin{remark}\label{rem-crit} 
Concerning the case $(\rho_-, \rho_+)(t) \in \Theta$, 
$Q_*(u, \rho_\pm(t))$ has opposite sign in $I_t$ except two zero points $\rho_\pm(t)$, 
therefore $\nu_{t,x}$ concentrates on $\{\rho_-(t), \rho_+(t)\}$. 
Suppose $f(t,x) = \nu_{t,x}(\rho_+(t))$, then 
\begin{equation}
\label{eq:critical}
  \nu_{t,x}(d\la) = [1-f(t,x)]\delta_{\rho_-(t)}(d\la) + f(t,x)\delta_{\rho_+(t)}(d\la). 
\end{equation}
Observing that $J(\rho_+) = J(\rho_-)$, so that 
\begin{equation}
  J(u^\ve(t,x)) \wstar \int_0^1 J(\la)\nu_{t,x}(d\la) = J(\rho_-(t)) = J(\rho_+(t)),
  \quad \ve \to 0, 
\end{equation}
as stated in Remark \ref{rem:current}. 
\end{remark}

\end{document}